\documentclass[12pt]{article}
\usepackage{amsmath,amsfonts,amsthm,amssymb}

\usepackage{amssymb,amsmath}
\parskip\medskipamount
\newtheorem{theorem}{Theorem}[section]

\newtheorem{lemma}[theorem]{Lemma}
\newtheorem{proposition}[theorem]{Proposition}
\newtheorem{corollary}[theorem]{Corollary}

\numberwithin{equation}{section}

\def\PG{{\sf PG}}

\def\D{\Delta} \def\t{\theta} 

\title{Automorphisms and opposition in twin buildings}
\author{Alice Devillers \and James Parkinson\footnote{Research supported under the Australian Research Council (ARC) discovery grant DP110103205.} \and Hendrik Van Maldeghem}
\date{}
\begin{document}
\maketitle

\begin{abstract}
We show that every automorphism of a thick twin building interchanging the halves of the building maps some residue to an opposite one. Furthermore we show that no automorphism of a locally finite 2-spherical twin building of rank at least 3 maps every residue of one fixed type to an opposite. The main ingredient of the proof is a lemma that states that every duality of a thick finite projective plane admits an absolute point, i.e., a point mapped onto an incident line. Our results also hold for all finite irreducible spherical buildings of rank at least 3, and as a consequence we deduce that every involution of a thick irreducible finite spherical building of rank at least 3 has a fixed residue. \let\thefootnote\relax
\footnote{2010 Mathematics Subject Classification: 20E42, 51E24}\footnote{Keywords: Twin buildings, projective planes, spherical buildings}
\end{abstract}

\section{Introduction}

Fixpoint structures of automorphisms play an important and prominent role in the theory of buildings and related groups. Recently it was shown in the papers \cite{TTV1,TTV2,TTV3,VM} that large fixpoint structures in spherical buildings are often implied by automorphisms that do not map any chambers to opposite chambers. In other words, if no chamber is mapped far away then the automorphism will fix a lot of simplices. In the present paper we consider complementary questions for twin buildings~$\Delta=(\Delta^+,\Delta^-,\delta^*)$. We consider automorphisms $\theta:\Delta\to \Delta$ which interchange the two halves of the building. There is quite a lot known when $\theta$ is assumed to be an involution (see \cite{DMGH}, \cite{GHM}, \cite{Horn}), but we do not make this assumption here.

\newpage

An initial observation is (c.f. \cite{AB2} for the spherical case, \cite{Horn} for $\theta$ an involution, and Section~\ref{section:Main0} for the general case): 

\begin{proposition}\label{Main0} Let $\theta$ be an automorphism of a thick twin building $\Delta$ swapping the two halves. Then there is a spherical residue of $\Delta$ which is mapped to an opposite residue.
\end{proposition} 

Thus it is natural to ask `How much can be mapped to an opposite?'. For example, in the real projective plane $\PG(2,\mathbb{R})$ the automorphism $\theta$ given by $[a:b:c]\leftrightarrow (a:b:c)$ (in homogeneous coordinates) maps \textit{every} flag to an opposite flag. Furthermore it is shown in \cite[Remark~4.5]{TTV2} that there are automorphisms of thick finite generalised quadrangles of order $(2^n-1,2^n+1)$ mapping \textit{every} flag to an opposite. Considered as twin buildings, these examples show that it is possible for an automorphism of a twin building to map every chamber to an opposite chamber. Other such examples can be constructed for the twin tree arising from $SL_2(\mathbb{F}[t,t^{-1}])$ where $\mathbb{F}$ is any field. Despite these examples our main theorem is:

\begin{theorem}\label{Main1}
Suppose that $\Delta$ is an irreducible $2$-spherical locally finite thick twin building of type $(W,S)$ with rank at least $3$ and let $J\subseteq S$ be nonempty. Then an automorphism $\theta$ cannot map every $S\backslash J$-residue to an opposite residue.
\end{theorem}

(In fact we prove a slightly stronger statement; see Theorem~\ref{thm:precise}. Note also that in the `simplicial complex' language $S\backslash J$-residues correspond to simplices of type $J$). 

In particular we see that under the hypothesis of Theorem~\ref{Main1} it is impossible for an automorphism $\theta$ to map every chamber to an opposite. After some general reductions, the proof of Theorem~\ref{Main1} boils down to looking at rank~2 residues. Recent results of \cite{TTV0} (see also \cite{T}) imply that finite Moufang generalized polygons other than projective planes do not admit collineations mapping every chamber to an opposite. For projective planes, the results of \cite{TTV0} can only be applied under some restrictions. In the present paper we remove these restrictions, and show that no duality of any finite projective plane can map every chamber to an opposite. These observations will imply Theorem~\ref{Main1}.

Along the way to proving Theorem~\ref{Main1} we prove the following theorem, which is of interest in its own right.

\begin{theorem}\label{Main2}
Let $\Delta$ be a twin building of type $(W,S)$ and let $J\subseteq S$ be nonempty. An automorphism $\theta:\Delta\to\Delta$ maps all $S\backslash J$-residues to opposite residues if and only if it maps all chambers to opposite chambers. Moreover, if $\theta$ maps all chambers to opposites then $\theta$ is necessarily type preserving, and  maps all residues to opposite residues.
\end{theorem}

\section{Definitions}\label{sect:definitions}

Let $(W,S)$ be a Coxeter system. We always assume that $|S|<\infty$, and that $W$ is irreducible. A \textit{twin building} $\Delta$ of type $(W,S)$ consists of two buildings $\Delta^+$ and $\Delta^-$ of type $(W,S)$ along with a $W$-valued \textit{codistance function} $\delta^*$ which measures codistance between chambers in the `two halves' of the twin building. This codistance function satisfies the following three axioms, see \cite[Definition~5.133]{AB} or \cite[\S 2.2]{T1}, where we write $\delta$ for the usual $W$-valued distance function on the buildings $\Delta^+$ and $\Delta^-$, and where $\ell(w)$ denotes the length of $w\in W$ with respect to $S$.

The following must hold for all $\epsilon\in\{+,-\}$, $C\in\Delta^\epsilon$ and $D\in\Delta^{-\epsilon}$, where $w=\delta^*(C,D)$: 
\begin{enumerate}
{\setlength\itemindent{10pt}\item[(Tw1)] $\delta^*(D,C)=w^{-1}$;}
{\setlength\itemindent{10pt}\item[(Tw2)] if $E\in\Delta^{-\epsilon}$ is such that $\delta(D,E)=s\in S$ and $\ell(sw)=\ell(w)-1$, then $\delta^*(C,E)=sw$;}
{\setlength\itemindent{10pt}\item[(Tw3)] if $s\in S$, there exists $E\in\Delta^{-\epsilon}$ such that $\delta(D,E)=s$ and $\delta^*(C,E)=sw$.}
\end{enumerate}

Chambers $C$ and $D$ in the same half of the building are \textit{$s$-adjacent}, $s\in S$, if $\delta(C,D)=s$. Chambers $C$ and $D$ are \textit{adjacent} if they are $s$-adjacent for some $s\in S$. A \emph{$J$-residue} of $\Delta^{\epsilon}$, $J\subseteq S$, is a set of the form $
\{D\in\Delta^{\epsilon}\mid\delta(C,D)\in W_J\}
$ where $C\in\Delta^{\epsilon}$ and 
where $W_J$ is the subgroup of $W$ generated by $J$. A $J$-residue is also called a \emph{residue} or a \emph{residue of type $J$}. There is a standard way to interpret residues as simplices, whereby a $J$-residue becomes a simplex of \textit{type} $S\backslash J$. In particular, $S\backslash\{s\}$-residues correspond to vertices of type~$s$. In this paper we will mainly use the `residue' language.

Chambers $C$ and $D$ in different halves of the twin building are \textit{opposite} if $\delta^*(C,D)=1$. Residues $P$ and $Q$ in different halves of the twin building are \textit{opposite} if for each chamber $C$ in $P$ there is a chamber $D$ in $Q$ such that $C$ and $D$ are opposite. Note that the axioms imply that opposite residues have the same type. 

An \textit{automorphism} of a twin building is a bijection $\theta$ of the chamber sets which maps adjacent chambers to adjacent chambers, and preserves the opposition relation. In this paper we are interested in automorphisms $\theta$ which interchange the two halves of the twin building. In general we do not assume that $\theta$ is an involution, or that it is type preserving.

By a standard gallery argument one sees that an automorphism $\theta:\Delta\to \Delta$ of a twin building induces an automorphism $\sigma:S\to S$ of the Coxeter graph such that, if the chambers $C$ and $D$ are $s$-adjacent, then $C^{\theta}$ and $D^{\theta}$ are $\sigma(s)$-adjacent.

We recall some more terminology: an $\{s\}$-residue, $s\in S$, is sometimes called an $s$-\emph{panel} or a \emph{panel} for short. A (twin) building is called \emph{thick} if every panel contains at least three chambers. A $2$-\emph{spherical} (twin) building is one where every rank~$2$ residue is spherical, i.e., no rank~$2$ residue is a tree.  Recall also that a \emph{locally finite} building is one where the number of chambers in a panel is always finite; locally finite is equivalent to finite for spherical buildings. In every thick twin building and in every thick $2$-spherical building, the number of chambers in an $s$-panel, $s\in S$, only depends on $s$.  The building is locally finite precisely when all these numbers are integers.

\section{Proof of Proposition~\ref{Main0}}\label{section:Main0}

The arguments of this section are adapted from \cite{AB2} (where it is shown that automorphisms of non-spherical buildings have unbounded displacement) and~\cite{Horn} (where involutions of twin buildings are considered). 

\begin{lemma}
Let $\sigma:S\to S$ be an automorphism of the Coxeter graph of $W$. Let $w\in W$, and let $J=\{s\in S\mid\ell(sw)<\ell(w)\}$. Then $W_J$ is spherical, and if $\ell(w\sigma(s))<\ell(w)$ and $\ell(sw\sigma(s))=\ell(w)$ for all $s\in J$, then $\sigma(J)=J$, $w=w_J$ (the longest element of $W_J$), and $sw\sigma(s)=w$ for all $s\in J$.
\end{lemma}

\begin{proof} By~\cite[Proposition~2.17]{AB} the parabolic subgroup $W_J$ is finite. 
Assume now that  $\ell(w\sigma(s))<\ell(w)$ and $\ell(sw\sigma(s))=\ell(w)$ for all $s\in J$. The first inequality tells us that $\sigma(J)\subseteq J'=\{t\in S|\ell(wt)<\ell(w)\}$.
By the dual version of~\cite[Proposition~2.17]{AB}, it follows that there is an expression $w=v'w_{\sigma(J)}$ with $\ell(w)=\ell(v')+\ell(w_{\sigma(J)})$, since $w_{\sigma(J)}\in W_{J'}$. If $v'\neq 1$ has reduced expression $v'=s_1\cdots s_{r}$, then $s_1\in J$, and since there is a reduced  expression for $w_{\sigma(J)}$ ending in $\sigma(s_1)$ we obtain a reduced expression for $w$ which starts with $s_1\in J$ and ends with $\sigma(s_1)$, contradicting the condition that for all $s\in J$, we require $\ell(sw\sigma(s))=\ell(w)$. Therefore $v'=1$, and $w=w_{\sigma(J)}$. Since $\sigma(s)w_{\sigma(J)}$ ($s\in J$) is in $W_{\sigma(J)}$, it must be shorter than  $w_{\sigma(J)}$, and so $\sigma(J)\subseteq J$. Since $J$ is finite and $\sigma$ is a bijection, it follows that $\sigma(J)=J$ and $w=w_J$. Now  $sw\sigma(s)\in W_J$ for $s\in J$ (since $s,\sigma(s)\in J$ and $w=w_J\in W_J$) and  $\ell(sw\sigma(s))=\ell(w_J)$, so  $sw\sigma(s)=w$ for all $s\in J$ (by the uniqueness of the longest element of $W_J$).
\end{proof}

\begin{proof}[Proof of Proposition \emph{\ref{Main0}}]
Let $C$ be a chamber such that the codistance $w=\delta^*(C,C^{\theta})$ has minimal length. Let $J=\{s\in S\mid \ell(sw)<\ell(w)\}$, and let $\sigma:S\to S$ be the automorphism of the Coxeter graph induced by~$\theta$. We claim that for all $s\in J$ we have $\ell(w\sigma(s))<\ell(w)$ and $\ell(sw\sigma(s))=\ell(w)$.

For if $\ell(w\sigma(s))>\ell(w)$ then, by thickness, we can choose a chamber $D$ with $\delta(C,D)=s$ such that $\delta^*(C,D^{\theta})=w$. Since $\ell(sw)<\ell(w)$ we have $\delta^*(D,D^{\theta})=sw$, contradicting minimality of~$w$.

Therefore $\ell(w\sigma(s))<\ell(w)$, and so for all $D$ with $\delta(C,D)=s$ we have $\delta^*(C,D^{\theta})=w\sigma(s)$. Therefore $\delta^*(D,D^{\theta})\in\{w\sigma(s),sw\sigma(s)\}$. By minimality of $w$ we have $\delta^*(D,D^{\theta})=sw\sigma(s)$ and $\ell(sw\sigma(s))=\ell(w)$.

By the previous lemma,  $w=w_J=sw\sigma(s)\in W_J$.
By connectivity in the $J$-residue $R$ of $C$, we get that  $\delta^*(E,E^{\theta})=w_J\in W_J$ for all chamber $E$ in $R$, and so every chamber in $R$ is opposite to a chamber in $R^\theta$, that is, $R$ is a spherical residue which is opposite its image under $\theta$.
\end{proof}

\section{Proof of Theorem~\ref{Main2}}

In this section we prove Theorem~\ref{Main2}, which will be an ingredient of the proof of Theorem~\ref{Main1}. We call an automorphism \textit{$J$-opposite} if it maps all $S\backslash J$-residues to opposite residues (in other words, $\theta$ maps all type $J$ simplices to opposites). We simply say \textit{opposite} instead of $S$-opposite (and so an automorphism is opposite if it maps all chambers to opposite chambers).

A key observation is that if $\theta$ is $J$-opposite then $\delta^*(C,C^{\theta})\in W_{S\backslash J}$ for all chambers~$C$. Indeed if the $S\backslash J$-residue of $C$ is opposite the $S\backslash J$-residue of $C^{\theta}$ then there is a chamber $D$ in the $S\backslash J$-residue of $C$ with $\delta^*(D,C^\theta)=1$. Since $\delta(C,D)\in W_{S\backslash J}$ it follows that we have $\delta^*(C,C^{\theta})\in W_{S\backslash J}$.

Theorem~\ref{Main2} follows immediately from Lemmas~\ref{lem1} and~\ref{lem2} below. 

\begin{lemma}\label{lem1}
Suppose that the automorphism $\theta$ of $\Delta$ is $J$-opposite with $J\subseteq S$ nonempty. Then $\theta$ is $\{s\}$-opposite for every $s\in J$.
\end{lemma}

\begin{proof}
We only need to show that $\theta$ maps $(S\setminus\{s\})$-residues (which we shall briefly call here type~$s$ vertices using the simplicial complex terminology) to type~$s$ vertices, for each $s\in J$. Let $x$ be a type $s$ vertex, with $s\in J$, and suppose that $x^{\theta}$ has type~$s'$. Then $s'\in J$ because $\theta$ preserves $J$ set-wise (as $\theta$ maps $S\backslash J$-residues to $S\backslash J$-residues). Let $C$ be a chamber contained in~$x$. Since $\theta$ is $J$-opposite we have $w:=\delta^*(C,C^{\theta})\in W_{S\backslash J}$. By (Tw3) there is a chamber $D$ with $\delta(C,D)=s$ and $\delta^*(D,C^{\theta})=sw$. Since $\theta$ maps the type $s$ vertex of $C$ to the type $s'$ vertex of $C^{\theta}$ we have $\delta(C^{\theta},D^{\theta})=s'$, and so \cite[Lemma~5.139]{AB} gives $\delta^*(D,D^{\theta})\in\{sw,sws'\}$. Since $\theta$ is $J$-opposite we have $\delta^*(D,D^{\theta})\in W_{S\backslash J}$. Since $w\in W_{S\backslash J}$ this forces $\delta^*(D,D^{\theta})=sws'$, and since this must be an element of $W_{S\backslash J}$ the Deletion Condition \cite[\S2.1]{AB} implies that $sws'=w$. Thus $sw=ws'$. The expressions $sw$ and $ws'$ are reduced since $w\in W_{S\backslash J}$, and so by \cite[Proposition~2.16]{AB} we have $s=s'$.
\end{proof}

\begin{lemma}\label{lem:coxeter}
If $w\in W_{S\backslash\{s\}}$, $s\in S$, then $sw=ws$ if and only if $s$ commutes with each generator appearing in a reduced expression for~$w$.
\end{lemma}

\begin{proof}
We have $\ell(ws)=\ell(w)+1$ because $w\in W_{S\backslash\{s\}}$. It follows from Tits' solution to the Word Problem \cite[Theorem~2.33]{AB} that every reduced expression for $ws$ has exactly~$1$ occurrence of the generator $s$, and that if $s'$ appears in a reduced expression for $w$, and if $s'$ does not commute with $s$, then every occurrence of $s'$ in any reduced expression for $ws$ must occur to the left of the unique $s$ generator.
\end{proof}

\begin{lemma}\label{lem2}
If $\theta$ is $\{s\}$-opposite for some~$s\in S$ then $\theta$ is $S$-opposite.
\end{lemma}

\begin{proof}
For the proof we define the \textit{Coxeter distance} $\mathrm{cox}(u,v)$ between $u\in W$ and $v\in W$ to be the minimum distance in the Coxeter graph of $W$ between nodes $s$ and $t$ such that $s$ appears in a reduced expression for $u$, and $t$ appears in a reduced expression for $v$. By \cite[Proposition~2.16]{AB}, this is well-defined.  We set the convention that $\mathrm{cox}(u,v)=\infty$ if either $u$ or $v$ is the identity. (By irreducibility, if $u,v\neq 1$ then $\mathrm{cox}(u,v)<\infty$.)

Suppose that $\theta$ is $\{s\}$-opposite. Let $C$ be a chamber of~$\Delta$ such that the Coxeter distance between $s$ and $w:=\delta^*(C,C^{\theta})$ is minimal. Since $w\in W_{S\backslash\{s\}}$ we have $\mathrm{cox}(s,w)\geq1$. We aim to show that $\mathrm{cox}(s,w)=\infty$ (and so $\delta^*(C,C^{\theta})=1$ for all chambers $C$, hence the result).

Suppose, for a contradiction, that $d=\mathrm{cox}(s,w)$ satisfies $1\leq d<\infty$. Let $s'$ be a generator appearing in a reduced expression for $w$ with $\mathrm{cox}(s,s')=d$. Let $t$ be the second last node on a geodesic in the Coxeter graph from $s$ to $s'$, so that $\mathrm{cox}(s,t)=d-1$ and $\mathrm{cox}(t,w)=\mathrm{cox}(t,s')=1$. By the twin building axiom (Tw3) we can choose a chamber $D$ with $\delta(C,D)=t$ and $\delta^*(D,C^{\theta})=tw$. Then $\delta^*(D,D^{\theta})\in\{tw,twt'\}$ where $\delta(C^{\theta},D^{\theta})=t'$. We consider each case.

Suppose that $\delta^*(D,D^{\theta})=tw$. Since $t$ does not appear in a reduced expression for $w$ we have $\ell(tw)=\ell(w)+1$, and therefore $\mathrm{cox}(s,tw)=d-1$, contradicting the fact that the chamber $C$ minimises Coxeter distance.

Suppose that $\delta^*(D,D^{\theta})=twt'$. Suppose first that $t'\neq t$ (this case only happens if $d>1$). Since $\theta$ induces an automorphism of the Coxeter graph preserving $s$ we have $\mathrm{cox}(s,t)=\mathrm{cox}(s,t')=d-1$. Thus $\ell(twt')=\ell(w)+2$ (since neither $t$ nor $t'$ appear in a reduced expression for $w$). Therefore $\mathrm{cox}(s,twt')=d-1$, a contradiction. So we must have that $t'=t$. Since $\ell(tw)=\ell(w)+1$ we have $\ell(twt)=\ell(w)+2$ or $\ell(twt)=\ell(w)$. If $\ell(twt)=\ell(w)+2$ then $\mathrm{cox}(s,twt)=d-1$, a contradiction. If $\ell(twt)=\ell(w)$ then by the Deletion Condition $twt=w$ (for otherwise $\mathrm{cox}(s,twt)=d-1$, a contradiction). Therefore $tw=wt$, contradicting Lemma~\ref{lem:coxeter} since $\mathrm{cox}(s,s')=1$ and so $s$ and $s'$ do not commute.

 Therefore $\mathrm{cox}(s,w)=\infty$, and the proof is complete.
\end{proof}

\begin{proof}[Proof of Proposition \emph{\ref{Main2}}]
 By Lemmas~\ref{lem1} and~\ref{lem2}, if  $\theta$  is $J$-opposite with $J\subseteq S$ nonempty then  $\theta$ is opposite and so it maps all chambers to opposite chambers.
If $\theta$ maps all chambers to opposite chambers, then it is immediate that  $\theta$  maps all residues to opposite residues (for $C$ containing a residue $R$, $C^\theta$ contains $R^\theta$ and is opposite to $C$), so $\theta$ is $J$-opposite for any $J\subseteq S$.  In particular, $\theta$ is $\{s\}$-opposite for all $s\in S$, which means vertices of type $s$ are mapped to vertices of type $s$ for all $s\in S$.

\end{proof}

\section{Proof of Theorem \ref{Main1}}\label{sect:Main1}

We call a rank~2 building \textit{non-exotic} if it is thick, finite and has parameters $(s,t)$ with $\gcd(s,t)>1$ (by \textit{parameters} $(s,t)$ we mean that every panel of one type contains precisely $s+1$ chambers while every panel of the other type contains precisely $t+1$ chambers). Usually rank 2 buildings are viewed as geometries, where one type of panels is the point set, and the other the set of lines. A line then carries $s+1$ points, and a point is incident with precisely $t+1$ lines. 

It follows from \cite[5.6 Corollary~3]{T1} that every rank~$2$ residue in a locally finite $2$-spherical twin building of rank at least~$3$ satisfies the \textit{Moufang condition}. Hence it follows from the last sentences of Section~3 of \cite{T0} and the papers \cite{Fon-Sei:73,Fon-Sei:74} that every rank 2 residue in a locally finite 2-spherical twin building of rank at least 3 is non-exotic. (In particular this applies to any finite spherical building of rank at least~$3$). Hence the following theorem, which we prove in this section, implies Theorem~\ref{Main1}.

\begin{theorem}\label{thm:precise}
Suppose that $\Delta$ is a twin building of type $(W,S)$ with at least one non-exotic rank~$2$ residue. Let $J\subseteq S$ be nonempty. Then an automorphism $\theta:\Delta\to\Delta$ cannot map every $S\backslash J$-residue to an opposite residue.
\end{theorem}

We will deduce Theorem~\ref{thm:precise} from Theorem~\ref{Main2} and the following propositions about automorphisms of rank~$2$ buildings. We will take the geometric view on buildings of rank 2, as noted above. In this point of view, rank~$2$ buildings are precisely the same as \textit{generalised $m$-gons} (bipartite graphs with diameter $m$ and girth~$2m$). By the Feit-Higman Theorem finite thick generalised $m$-gons only exist for $m=2,3,4,6,8$. Generalised $m$-gons  with $m=4,6,8$ are called generalised \textit{quadrangles}, \textit{hexagons}, and \textit{octagons} (respectively). 

Recall that a \textit{collineation} (respectively \textit{duality}) of a generalised $m$-gon $\Delta$ is a type preserving (respectively type swapping) automorphism $\theta:\Delta\to\Delta$.  Here, the type is with respect to the `single' spherical building structure, and not as a twin building (see \cite[Example~5.136]{AB}). Thus a duality maps points to lines. An \textit{absolute point} of a duality of a projective plane is a point which is mapped to a line incident with the point.

We record a result from~\cite{TTV0} (see also~\cite{T}).

\begin{proposition}[\cite{TTV0} and \cite{T}]\label{Beukjeeven}
Let $\Delta$ be a finite generalised quadrangle, hexagon or octagon with parameters $(s,t)$, where $\gcd(s,t)>1$. Then every collineation of $\Delta$ maps some point to a point at distance at most $2$ in the incidence graph. In particular, no collineation is opposite. 
\end{proposition}

Dualities of finite projective planes are also treated in \cite{TTV0}, but there is an error in the proof of Corollary~5.5, which is pointed out in \cite{TTV0'}. The correct version can be found in \cite{T}, Corollary~1.4.5, which we restate here.

\begin{proposition}\label{Beukje}
Let $\Delta$ be a finite projective plane of order $q$ and let $\t$ be a duality of $\Delta$ of order $n$. Let $q'$ be the square-free part of $q$ (with $q'=0$ if $q$ is a perfect square). Then $\t$ admits at least one absolute point if one of the following conditions is satisfied.
\begin{enumerate}\item[$(i)$] $q'$ does not divide $n$; \item[$(ii)$] $q'$ is even and divides $n$, but $8$ does not divide $n$; \item[$(iii)$] $q'=3~\mbox{mod }4$, $q'$ divides $n$, but $4$ does not divide $n$.\end{enumerate}  
\end{proposition}

In the next proposition we show that these conditions can be omitted (not only for the classical finite projective planes, but for all finite projective planes including potential projective planes of non-prime-power orders). 

\begin{proposition}\label{dualitiesareisotropic}
Any duality of a finite projective plane has an absolute point. 
\end{proposition}
\begin{proof}
Suppose this is not true and consider a smallest counter-example $\D$, of order $q$, with a duality $\t$ without absolute points.

By Proposition~\ref{Beukje}, we know that every duality of a plane of square order $q$ admits at least one absolute point, so $q$ is not a square. Let $q'\neq 0$ be its square-free part. Also by Proposition~\ref{Beukje}, we know that, if $q'$ does not divide $|\t|$ (the order of $\t$), then $\t$ admits at least one absolute point. Hence we may assume from now on that $q'$ divides  $|\t|$. Now we claim that $|\t|$ can be written as $2q'r$, $r\in\mathbb{N}$. Indeed, if $q'$ is odd, this follows from the fact that $\t$ is a duality. If $q'$ is even, then by Proposition~\ref{Beukje}, $|\t|$ is a multiple of $8$ and the claim follows. 

Let $\t'=\t^{q'r}$. Then $\t'$ is an involution. We now divide our arguments in three cases.

\medskip

\fbox{Case 1: $q'r$ is even and $q$ is even.} 

In this case $\t'$ is a collineation, i.e., a type-preserving automorphism. 

Now we will use the fact that there are only three possible types of collineations of order 2 of a projective plane of order $q$: 
\emph{homologies, elations}, and \emph{Baer involutions}, see Theorems~4.3 and~4.4 of \cite{HP}. They are characterised by their 
set of fixed point and lines. A Baer involution happens only for $q$ a perfect square and it pointwise fixes a Baer subplane (a 
projective plane of order $\sqrt{q}$). Homologies and elations happen  when $q$ is odd and even, respectively, and they are 
central collineations, that is, they have a unique centre (all lines through the centre are fixed), and a unique axis (all points 
on the axis are fixed).
For homologies the centre is not on the axis, for elations it is. 

Hence in our present situation, $\t'$ is an elation with axis $L$ and centre $x$, where $x\in L$. Since $\t$ centralizes $\t'$, $\t$ 
acts on the set of fixed points and fixed lines of $\t'$, and hence must fix the flag $\{x,L\}$. Therefore $x$ is an absolute point for 
$\t$, a contradiction. 

\medskip

\fbox{Case 2: $q'r$ is even and $q$ is odd.}

In this case $\t'$ is again a collineation of order $2$, but since $q$ is odd, $\t'$ is a homology with axis $L$ and centre $x$, where $x\notin L$.  Since $\t$ centralizes $\t'$, $\t$ must interchange $x$ and $L$, and preserve the set $\mathcal{F}$ of flags $\{y,yx\}$ where $y\in L$ and $xy$ denotes the line containing $x$ and $y$. Since $q$ is odd, so is $q'$. Pick a prime $p$ dividing $q'$ and write $|\t|=p^h\ell$, where $\ell$ is not divisible by $p$. Note that $\ell$ is even. Let $\t''=\t^\ell$. It has order $p^h$ and is a collineation. Since $\t''$ centralises $\t'$, it fixes both $x$ and $L$, and preserves the set $\mathcal{F}$. All the orbits of $\t''$ have size $1$ or a power of $p$. Since  $\mathcal{F}$ has size $q+1$, which is congruent to $1$ modulo $p$, $\t''$ must fix at least one flag of $\mathcal{F}$. If  there were only one such flag, then it follows from the fact that $\t$ centralises $\t''$ that $\t$ would fix that flag, and hence the point of that flag is absolute, a contradiction. 

 Let $\mathcal{P}$ be the set of fixed points of $\t''$ and let $\mathcal{L}$ be the set of lines of $\D$ intersecting $\mathcal{P}$ in at least two points.
We claim that $\Delta'=(\mathcal{P},\mathcal{L})$ is a projective plane. By definition, two points are on one line, and since the intersection of two lines with two fixed points each is also fixed, two lines in $\mathcal{L}$ meet in a point of $\mathcal{P}$. The only remaining axiom to check is that $(\mathcal{P},\mathcal{L})$ contains a quadrangle.
 Let $\{z_i,M_i\}$, $i=1,2$, be two flags of $\mathcal{F}$ fixed by $\t''$. Since $\t''$ fixes at least two points of $M_i$, namely $x$ and $z_i$ and all $\t''$-orbits have size $1$ or a power of $p$, there are at least $p+1$ points of $M_i$ in $\mathcal{P}$. Let $y_i$ be a point of  $M_i$ distinct from $x$ and $z_i$, in $\mathcal{P}$ ($i=1,2$). Then $\{y_1,z_1,y_2,z_2\}$ forms a quadrangle.
Hence $\D'=(\mathcal{P},\mathcal{L})$ is a projective plane. Since $\t''$ is not the identity, $\mathcal{P}$ is strictly contained in the pointset of $\D$. 

Since $\t$ centralises $\t''$, $\t$ acts on the projective plane $\D'$, also without absolute points. This contradicts the fact that  $\D$ was a smallest counter-example. 

\medskip

\fbox{Case 3: $q'r$ is odd.} 

In this case $\t'$ is a duality of order $2$, i.e., a polarity, of a projective plane with non-square order $q$.  Baer showed \cite{Ba46} that the number of absolute points 
of $\t'$ is exactly $q+1$. Moreover, for $q$ even, all the absolute points are collinear; for $q$ odd, not more than two absolute points lie on a given line.  
Note that the point $x$ is absolute for $\t'$ if and only if $\{x,x^{\t'}\}$ is a flag fixed by $\t'$.
Let $\mathcal{F}$ be the set of flags fixed by $\t'$; this set has size $q+1$. 

Pick a prime $p$ dividing $q'$ and write $|\t|=p^h\ell$, where $\ell$ is not a multiple of $p$ (note that $p$ is odd and $\ell$ is even). Let $\t''=\t^\ell$. It has order $p^h$ and is a collineation. Since $\t''$ centralises $\t'$, it  preserves the set $\mathcal{F}$. All the orbits of $\t''$ have size $1$ or a power of $p$. Since  $\mathcal{F}$ has size $q+1$, which is congruent to $1$ modulo $p$, $\t''$ must fix at least one flag of $\mathcal{F}$. If there were only one such flag, then it follows from the fact that $\t$ centralises $\t''$ that $\t$ would fix that flag too, and then the point of that flag would be an absolute point for $\t$, a contradiction. Hence  $\t''$ must fix at least $p+1\geq 4$ flags of $\mathcal{F}$.

Let $\mathcal{P}$ be the set of fixed points of $\t''$ and let $\mathcal{L}$ be the set of lines of $\D$ intersecting $\mathcal{P}$ in at least two points.
We claim that $\D'=(\mathcal{P},\mathcal{L})$ is a projective plane. As above, the only significant axiom to check is that $(\mathcal{P},\mathcal{L})$ contains a quadrangle.
If $q$ is odd, then the flags $\{x,x^{\t'}\}$ of $\mathcal{F}$ are such that no three of the absolute points are on a given line, and since there are at least 4 such flags, $(\mathcal{P},\mathcal{L})$ contains a quadrangle.

So assume now that $q$ is even. Then, as noticed above, all the absolute points of $\t'$ form a line, say $L$. In other words, all the points of the flags of $\mathcal{F}$ lie on $L$, and there are at least $p+1\geq 4$ such points which are in $\mathcal{P}$. Since $L$ is the image by $\t'$ of at most one of them (namely $L^{\t'}$, if it is an absolute point), there are at least two points $x_1$, $x_2$ of $\mathcal{P}$ which are on $L$, and such that $x_i^{\t'}\neq L$ and $\{x_i,x_i^{\t'}\}\in\mathcal{F}$ is fixed by $\t''$,   $i=1,2$. Hence $\t''$ fixes $y:=x_1^{\t'} \cap x_2^{\t'}$.
Since $\t''$ fixes at least two points of $x_i^{\t'}$, namely $x_i$ and $y$ and all $\t''$-orbits have size $1$ or a power of $p$, there are at least $p+1$ points of $x_i^{\t'}$ in $\mathcal{P}$. Let $y_i$ be a point of  $x_i^{\t'}$ distinct from $x_i$ and $y$, in $\mathcal{P}$ ($i=1,2$). Then $\{x_1,y_1,x_2,y_2\}$ forms a quadrangle.

Hence, in all cases, $\D':=(\mathcal{P},\mathcal{L})$ is a projective plane. Since $\t''$ is not the identity, $\mathcal{P}$ is strictly contained in the pointset of $\D$. 

Since $\t$ centralises $\t''$, $\t$ acts on the projective plane $\D'$ and as above it has no absolute points, contradicting the fact that  $\D$ was  a smallest counter-example. 
\end{proof}

We now give the proof of Theorem~\ref{thm:precise}.

\begin{proof}[Proof of Theorem \emph{\ref{thm:precise}}] Suppose that $\theta:\Delta\to\Delta$ maps all $S\backslash J$-residues to opposites. Then by Theorem~\ref{Main2}, $\theta$ is type preserving and maps all chambers to opposites. Let $R$ be a non-exotic rank 2 residue of $\Delta$. Then $(R,R^\t)$ forms a twin building with codistance induced by the codistance of $\D$. Since $R$ is of spherical type we can see this twin building as a (single) spherical building $R$ in the usual way. The induced automorphism $\tilde{\t}:R\to R$ is a duality if $R$ is a projective plane, and a collineation for generalised quadrangles, hexagons, and octagons. This automorphism maps all chambers to opposites, contradicting Proposition~\ref{Beukjeeven} and 
Proposition~\ref{dualitiesareisotropic}. 
\end{proof}

We conclude with the following corollary.

\begin{corollary}\label{Main3}
An involution $\theta$ of a spherical building $\Delta$ either maps 
every chamber to an opposite or fixes at least one simplex. In 
particular, every involution of a finite irreducible thick 
spherical building of rank at least $3$ fixes some simplex.
\end{corollary}

\begin{proof}
Suppose that $\theta$ does not map every chamber to an opposite and 
let $C$ be a chamber such that $C^\theta$ is not opposite $C$. 
Suppose $C\neq C^\theta$. Choose an apartment $A$ through 
$C$ and $C^\theta$ and let $\theta'$ be the unique type preserving isomorphism 
$A^\theta\rightarrow A$ fixing $C$ and $C^{\theta}$. Then $\theta\theta'$ restricted to $A$ is an involution $A\to A$, as $(\theta\theta')^2$  fixes $C$ and $C^{\theta}$ but $\theta\theta'$ switches them. 

Let $\Sigma$ be the geometric realisation (on a sphere) of $A$ as a Coxeter complex. Consider the geodesic joining the (barycentres) of 
the chambers $C$ and $C^{\theta}$. Since $C$ and $C^{\theta}$ are not opposite, this geodesic is unique. In this case $\theta\theta'$ 
fixes the midpoint of the geodesic, and since $\theta'$ acts as the identity on the convex closure of $C$ and $C^\theta$, this implies 
that $\theta$ fixes some simplex.

If $\Delta$ is a finite thick spherical building of rank at least $3$, then all rank~$2$ residues are non-exotic (as we noted in the beginning of this section). By Theorem~\ref{Main1} the involution $\theta:\Delta\to \Delta$ 
cannot map all chambers to opposites, and hence fixes a simplex by the above argument. 
\end{proof}

\newpage
Addresses of the authors:

Alice Devillers\\
School of Mathematics and Statistics\\
The University of Western Australia\\
35 Stirling Highway\\
Crawley Perth WA 6009\\ 
\texttt{alice.devillers@uwa.edu.au}

James Parkinson\\
School of Mathematics and Statistics\\
The University of Sydney\\
NSW, 2006, Australia\\
\texttt{jamesp@maths.usyd.edu.au}

Hendrik Van Maldeghem\\
Department of Mathematics\\
Ghent University\\
Krijgslaan 281, S22,\\
9000 Gent, Begium\\
\texttt{hvm@cage.UGent.be}

\end{document}